\documentclass[12pt]{amsart}
\usepackage{amsthm,amsmath,amsfonts,amssymb,latexsym,amscd,mathrsfs,hyperref,cleveref}
\usepackage{graphicx}
\usepackage{amsaddr}
\usepackage{cite}

\theoremstyle{plain}
\newtheorem{thm}{Theorem}
\newtheorem{lem}{Lemma}

\theoremstyle{definition}

\newtheorem*{exa*}{Example}

\crefname{thm}{theorem}{theorems}
\crefname{lem}{lemma}{lemmas}

\newcommand{\bs}{\boldsymbol}
\newcommand{\mb}{\mathbb}
\newcommand{\mc}{\mathcal}

\renewcommand{\mod}{\operatorname{mod}}

\def \a{\alpha} \def \b{\beta} \def \d{\delta} \def \e{\varepsilon} \def \g{\gamma}   \def \s{\sigma} \def \t{\theta} 

\numberwithin{equation}{section}

\renewcommand{\labelenumi}{\setlength{\labelwidth}{\leftmargin}
   \addtolength{\labelwidth}{-\labelsep}
   \hbox to \labelwidth{\theenumi.\hfill}}

\begin{document}
\title{Fractional parts of polynomials over the primes}
\author{Roger Baker}

\address{\rm Dedicated to the memory of Klaus Roth.}

 \begin{abstract}
Let $f$ be a polynomial of degree $k > 1$ with irrational leading coefficient. We obtain results of the form
 \[
\|f(p)\| < p^{-\s}
 \]
for infinitely many primes $p$ that supersede those of Harman (1981, 1983) and Wong (1997).
 \end{abstract}
 


\maketitle

\section{Introduction}\label{sec1}

For $k \ge 2$, let $\rho_k$ denote the supremum of positive numbers $\nu$ for which
 \[
\|\a p^k + \b\| < p^{-\nu}
 \]
has infinitely many solutions in primes $p$ for every irrational $\a$ and real $\b$. Let $\s_k$ denote the supremum of positive numbers $\nu$ for which
 \begin{equation}\label{eq1.1}
\|f_k(p)\| < p^{-\nu}
 \end{equation}
has infinitely many solutions in primes $p$ whenever $f_k$ is a polynomial of degree $k$ with irrational leading coefficient. (See Matomaki \cite{mat} for the case $k=1$, which presents different features from $k=2, 3, \ldots$ .)

To state our main result we define the integer $J=J(f_k)$ as follows. For
 \begin{align*}
f_k(x) &= \a x^k + \b,\\
J(f_k) &= 2^{k-1}\ (k \le 5), \quad J(f_k) = k(k-1) \quad (k \ge 6)\\
\intertext{For other polynomials of degree $k$, let}
J(f_k) &= 2^{k-1}\ (k \le 7), \quad J(f_k) = 2k(k-1) \quad (k\ge 8).
 \end{align*}

 \begin{thm}\label{thm1}
The inequality \eqref{eq1.1} has infinitely many solutions for
 \begin{align*}
\nu &< \frac 2{13} \quad (k=2)\\[3mm]
\nu &< \frac 1{10} \quad (k=3)\\[3mm]
\nu &< \frac{0.4079}{J(f_k)} \quad (k\ge 4)
 \end{align*}
 \end{thm}

A few remarks are in order. Our first few results are of the form $\s_2 \ge 2/13$, $\s_3 \ge 1/10$, $\s_4 \ge 0.0509875$. Harman \cite{har3} obtained $\rho_2 \ge 2/13$, so we do not have a new result for the polynomials $\a x^2 + \b$. Wong \cite{wong} obtained lower bounds for $\rho_3, \ldots,\rho_{11}$ with $\rho_3 \ge 5/56 = 0.0892\ldots$ and $\rho_4 \ge 1/21 = 0.0476\ldots$. In \cite{har} Harman shows that $\s_k \ge 1/2^{2k-1}$, so that $\s_2 \ge 1/8$ and $\s_3 \ge 1/32$. Harman \cite{har2} gives improvements for $\s_4, \s_5,\ldots$ including $\s_4 \ge 4/391 = 0.0102\ldots$. Asymptotically, Harman \cite{har2} shows that
 \[
\s_k \ge \frac{1 + o(1)}{12k^2 \log k}\, .
 \]

Our improvements depend on obtaining new `arithmetical information' to use in the Harman sieve \cite{har3, har4}. This amounts to giving upper bounds for trilinear exponential sums, of the form
 \begin{equation}\label{eq1.2}
\sum_{\ell \le L} \ c_\ell \mathop{\sum_{X < x \le 2X}\ a_x \sum_{Y < y \le 2Y}}_{N/2 < xy < N} b_y e(\ell g(mn)) \ll N^{1-\eta},
 \end{equation}
where $L= N^{\rho-\e/3}$ and either $|a_x| \le 1$, $|b_y| \le 1$ (Type II sums), or $|a_x| \le 1$ and $b_y = 1$ identically (Type I sums). The point is to get the estimate over wider ranges than can be found in Baker and Harman \cite{bh}, the present `state of the art' for monomials. Here $g$ is obtained from $f_k$ by replacing its leading coefficient $\a_k$ by $a/q$, a convergent to $\a_k$. Several devices come into play. We give a sharper bound in an auxiliary result on the number of solutions $y \in (Y, 2Y]$ of
 \[
\left\|\frac{say^3}q\right\| < \frac 1Z
 \]
for a given integer $s < q$ by slightly adapting a result of Hooley \cite{hool}. For $k \ge 3$, we give a relatively simple argument that improves the lower bound on $Y$ in \eqref{eq1.2} from $Y \gg L^2N^{2\eta}$ (essentially) to $Y \gg L N^{2\eta}$, in the `Type II' case. For $k \ge 6$, we use stronger results on simultaneous approximation to the coefficients of a large Weyl sum \cite{rcb3} than those available to the authors of \cite{bh}; these simultaneous approximation results depend on the work of Bourgain, Demeter, and Guth \cite{bdg}.

We end this section with remarks on notation. We write `$y \sim Y$' for `$Y \le y < 2Y$'. We write $\langle s_1, \ldots, s_k\rangle$, $[s_1, \ldots, s_k]$ for greatest common divisor and least common multiple. Let $e(\t) = e^{2\pi i\t}$ and $\|x\| = \min\limits_{n\in \mb Z}|x-n|$. Constants implied by `O' and `$\ll$' depend only on $k$, $\e$. We suppose that the positive number $\e$ is sufficiently small and let $\eta = \e/C_1(k)$ where $C_1(k)$ is a suitable large positive constant. Let $a/q$ be a convergent (with $q$ sufficiently large) to the continued fraction of $\a_k$, where $f_k(x) = \a_k x^k + \cdots + \a_1x + \a_0$, and fix $N$ with
 \[
(L_1N^k)^{1/2} \ll q \ll (L_1N^k)^{1/2},
 \]
where $L_1$ denotes $2N^{\rho-\e/2}$ with $\rho$ defined by
 \[
\rho = \frac 2{13}\, (k=2), \ \rho = \frac 1{10}\, (k=3), \ \rho = \frac{0.4079}{J(f_k)}\, (k \ge 4).
 \]
Clearly (much as in \cite{bh}) it suffices to prove that there is a positive number of primes in the set
 \[
A = \left\{\frac N2 < n \le N : \|g(n)\| < L_1^{-1}\right\}.
 \]

Note that our definition of $L$ gives $L > N^{\e/7}L_1$; this `increase' compared to $L_1$ is required at the last stage of \Cref{lem11} below.
\vskip .3in

\section{Small values of a monomial $\pmod q$.}\label{sec2}

Let $1 \le Y < q$, $1 \le D < q$, $Z \ge 2$, $1 \le s < q$. For later use we need to bound the number of solutions of the inequality
 \begin{equation}\label{eq2.1}
\left\|\frac{say^k}q\right\| < \frac 1Z
 \end{equation}
for which $y \in (Y, 2Y]$ and $\langle y, q\rangle \le D$. Denote this number by $\mc N_k(Y, D, Z, s)$.

 \begin{lem}\label{lem1}
 \begin{enumerate}
\item[(i)] With the above notations, we have
 \[
\mc N_k(Y, 1, Z, s) \ll q^{1+\eta} Z^{-1}.
 \]

\item[(ii)] Whenever $s D^k < q$, we have
 \[
\mc N_k(Y, D, Z, s) \ll q^{1+2\eta} Z^{-1}.
 \]
 \end{enumerate}
 \end{lem}
 
 \begin{proof}
For part (i) see \cite[Lemma 6]{bh}. To deduce part (ii) it suffices to show that for $d\mid q$ the number of $y \sim Y$ with \eqref{eq2.1} and $\langle y,q\rangle = d \le D$ is
 \[
\ll q^{1+\eta} Z^{-1}.
 \]
Write $y = y_1d$, $q = q_1d$, $\langle y_1, q_1\rangle = 1$. Then \eqref{eq2.1} implies
 \[
\left\|\frac{sd^{k-1}y_1^k}{q_1}\right\| < Z^{-1}
 \]
and $y_1 \sim \frac Yd < q_1$. Since $sd^{k-1} < q_1$, the desired bound follows from part (i).
 \end{proof}
 
 \begin{lem}\label{lem2}
Let $Y$, $Z$ be positive numbers in $[1, N^3]$. Then
 \begin{align*}
\mc N_3(Y, q, Z, s) \ll Y^{1/2} + N^\eta \Bigg(YZ^{-1/4} &+ Y\left(\frac{\langle s,q\rangle}q\right)^{1/4}\\[2mm]
&\quad + Y^{1/4} q^{1/4} Z^{-1/4}\Bigg).
 \end{align*} 
 \end{lem}
 
 \begin{proof}
In the case $\langle s,q\rangle=1$, this follows from Hooley \cite[Theorem 1]{hool}. For the general case, we rewrite \eqref{eq2.1} as
 \[
\left\|\frac{s_1ay^3}{q_1}\right\| < \frac 1Z
 \]
where $d = \langle s,q\rangle$, $s = s_1d$, $q = q_1d$, $\langle s_1, q_1\rangle = 1$.
 \end{proof}
 
 \begin{lem}\label{lem3}
 \begin{enumerate}
\item [(i)] Let $s$, $Y$ be positive integers less than $q$ and let $Z \ge 2$. Then
 \[
\mc N_2(Y, 1, Z, s) \ll q^\eta \left(\frac{Y + q^{1/2}}{Z^{1/2}}\right).
 \]

\item[(ii)] Let $D \ge 1$. Whenever $s D^2 < q$, in addition to the above hypotheses, we have
 \[
\mc N_2(Y, D, Z, s) \ll q^{2\eta} \left(\frac{Y + q^{1/2}} {Z^{1/2}}\right).
 \]
 \end{enumerate}
 \end{lem}
 
 \begin{proof}
For (i), see \cite[Lemma 9]{bh}. We deduce (ii) from (i) by an argument used in proving \Cref{lem1}.
 \end{proof}

Our next task is to `allow $s$ to vary' in the counting performed in Lemmas \ref{lem1}--\ref{lem3}. For $S_0 \ge 1$, $S_1 \ge 1$ it is convenient to write
 \begin{align*}
\mc A&(S_0, S_1, d_0, d_1)\\
&\qquad =\{s = s_0s_1 : \langle s_0,s_1\rangle = 1, s_0 \text{ square-full}, S_0 < s_0 \le 2S_0,\\
&\hskip 1.3in s_1 \text{ square-free, } S_1 < s_1 \le 2S_1, d_0\mid s_0, d_1\mid s_1\}
 \end{align*}
whenever $d_0$, $d_1$ are positive integers.

 \begin{lem}\label{lem4}
For $S_0 \le N$, $S_1\le N$, 
 \[
\#\mc A(S_0, S_1, d_0, d_1) \ll N^\eta S_0^{1/2} S_1 d_0^{-1/2} d_1^{-1}.
 \]
 \end{lem}
 
 \begin{proof}
The number of possible $s_1$ here is $\ll S_1 d_1^{-1}$. Any $s_0$ occurring can be written
 \begin{equation}\label{eq2.2}
s_0 = d_0uv
 \end{equation}
where $p\mid u$ implies $p\mid d_0$, and $v$ is squarefull and relatively prime to $d_0$.

Obviously, $v \le 2S_0/d_0$, so there are $O(S_0^{1/2} d_0^{-1/2})$ possible $v$.

It remains to show that $H$, the number of different $u$ that can occur in \eqref{eq2.2}, is $O(N^\eta)$.

Now, writing $p_1 < \cdots < p_t$ for the prime divisors of $d_0$, we find that $H$ is at most equal to the number of tuples $(m_1, \ldots, m_t)$, $m_i$ a non-negative integer, with
 \begin{equation}\label{eq2.3}
m_1\log p_1 + \cdots + m_t \log p_t \le \log (2S_0).
 \end{equation}
A little thought (replace $p_1, \ldots, p_t$ in \eqref{eq2.3} by the first $t$ primes $q_1,\ldots, q_t$) shows that
 \[
H \le \Psi(2S_0, q_t),
 \]
in the usual notation for smooth numbers. Since $q_t < (1 + \e) t\log t$ if $t$ is large, and $t < (1+\e) \log N/\log\log N$, we have
 \[
H \le \Psi(2N, 2\log N).
 \]
An appeal to Theorem 1 of de Bruijn \cite{dbr} now yields
 \[
H \ll N^\eta,
 \]
and the lemma follows.
 \end{proof}

Let
 \begin{align*}
\mc M_k(Y, Z, S_0, S_1) &= \#\Bigg\{y \sim Y: \langle y, q\rangle 
 \le N^\rho, \left\|\frac{say^k}q\right\| < \frac 1Z\\[2mm]
&\qquad \text{for some } s \in \mc A(S_0, S_1, 1, 1)\Bigg\}.
 \end{align*}
Summing over $s$ in \Cref{lem1} and \Cref{lem3}, we obtain
 \begin{equation}\label{eq2.4}
\mc M_k(Y, Z, S_0, S_1) \ll N^{3\eta} S_0^{1/2} S_1\, \frac qZ
 \end{equation}
whenever $1 \le Y < q$, $Z \ge 2$ and
 \begin{equation}\label{eq2.5}
4S_0 S_1 N^{k\rho} < q, 
 \end{equation}
while under the same conditions on $Y$, $Z$, $S_0$, $S_1$,
 \begin{equation}\label{eq2.6}
\mc M_2(Y, Z, S_0, S_1) \ll N^{3\eta} S_0^{1/2} S_1(Y + q^{1/2}) Z^{-1/2}.
 \end{equation}
To obtain a bound for $\mc M_3(Y, Z, S_0, S_1)$, we restrict $s = s_0s_1$ to values with $\langle s_0, q\rangle = d_0$, $\langle s_1, q\rangle = d_1$, at a cost of a factor $O(N^\eta)$. Now \Cref{lem2} and \Cref{lem4} together yield
 \begin{equation}\label{eq2.7}
\mc M_3(Y, Z, S_0, S_1) \ll N^{3\eta} S_0^{1/2} S_1(Y^{1/2} + Y Z^{-1/4} + Y q^{-1/4} + Y^{1/4} q^{1/4} Z^{-1/4})
 \end{equation}
(the factor $(d_0d_1)^{1/4}$ in the third term in the bound in \Cref{lem2} is cancelled by the factor $d_0^{-1/2}d_1^{-1}$ in \Cref{lem4}).

In our applications we shall always have \eqref{eq2.5}. If we assume this additional condition for $k=3$, there are no solutions of
 \[
\left\|\frac{say^3}q\right\| < \frac 1q \quad \text{with} \ \langle y,q\rangle \le N^\rho
 \]
counted in \eqref{eq2.7}. So we may suppose that $Z < q$, and we obtain
 \begin{equation}\label{eq2.8}
\mc M_3(Y, Z, S_0, S_1) \ll N^{3\eta} S_0^{1/2} S_1(Y^{1/2} + Y Z^{-1/4} + Y^{1/4} q^{1/4} Z^{-1/4}).
 \end{equation}
 \vskip .3in

\section{Type I sums}\label{sec3}

Our most basic tool is obtained by combining Theorem 5.1 of \cite{rcb1} (with a correction in \cite{rcb2}) and Theorem 4 of \cite{rcb3}.
 
 \begin{lem}\label{lem5}
Let $f$ be a polynomial of degree $k$, $f(x) = \g_kx^k + \cdots + \g_1x + \g_0$. Let $M \ge 1$ and $X \ge 1$, with $M = 1$ when $J = J(f) \ne 2^{k-1}$. Suppose that for some subinterval $I$ of $[\frac X2, X]$ we have
 \[
 \sum_{m=1}^M \Bigg|\sum_{x\in I} e(mf(x))\Bigg|
 \ge P \ge MX^{1-1/J + \eta}.
 \]
Then there are natural numbers $s$ and integers $u_1, \ldots, u_k$ with $\langle s, u_1, \ldots, u_k\rangle = 1$; $\langle s, u_2,\ldots, u_k\rangle \le MX^\eta$, in the case $J(f) = 2^{k-1}$;
 \begin{gather*}
s \ll (MXP^{-1})^k X^\eta,\\[2mm]
|s\g_j - u_j| \ll M^{-1}(MXP^{-1})^kX^{\eta-j} \quad (j = 1, 2, \ldots, k).
 \end{gather*}
 \end{lem}
 
 \begin{lem}\label{lem6}
Let $u \in \mb Z$, $d \in \mb N$, $B \ge 1$, $L \ge 1$. Let $\mc N$ be the number of solutions of
 \[
\ell u \equiv b\pmod d \quad (1 \le \ell \le L, 1\le b \le B)
 \]
Then
 \[
\mc N \le \min(L, B) + \frac{BL}d\, .
 \] 
 \end{lem}

 \begin{proof}
The congruence has no solution unless $\langle u, d\rangle\mid b$. For fixed $b$, the number of possibilities for $\ell\left(\mod \frac d{\langle u, d\rangle}\right)$ is at most 1. Hence
 \[
\mc N \le \frac B{\langle u, d\rangle}\, \left(\frac{L\langle u, d\rangle}d + 1\right) \le \frac{BL}d + B.
 \]
On the other hand, for given $\ell$, the number of possible $b$ is at most $\frac Bd + 1$. This gives the alternative upper bound $\frac{BL}d + L$.
 \end{proof}
 
 \begin{lem}\label{lem7}
Let $k \ge 2$. Let $f(x) = \g_k x^k + \cdots + \g_1 x$. Let $X \ge 1$, $1 \le L\le X$. Suppose there are integers $s$, $u_1, \ldots, u_k$, $s \le X$, $\langle s, u_1, \ldots, u_k\rangle = 1$, and if $J(f) = 2^{k-1}$, $\langle s, u_2, \ldots, u_k\rangle \ll LN^\eta$, such that
 \begin{equation}\label{eq3.1}
|s\g_j - u_j| \le (2k^2)^{-1} L^{-1} X^{1-j}\quad (1 \le j \le k). 
 \end{equation}
Let
 \begin{gather*}
\b_j = \g_j - \frac{u_j}s, \ F(x) = \sum_{j=1}^k \b_j x^j, \ G(x)= \sum_{j=1}^k u_jx^j,\\[2mm]
S(s, \ell G) = \sum_{v=1}^s e\left(\frac{\ell G(v)}s\right). 
 \end{gather*}
Then we have, for any subinterval $I$ of $\left[\frac X2, X\right]$,
 \begin{align}
\sum_{\ell = 1}^L \Bigg| \sum_{n\in I} e(\ell f(x)) &- s^{-1} S(s, \ell G) \int_I e(\ell F(z)) dz\Bigg|\\[4mm]
&\ll \begin{cases}
 N^{2\eta} Ls^{1-1/k} & \text{if } J(f) = 2^{k-1}\label{eq3.2}\\[2mm]
 N^{2\eta}(Ls^{1-1/k}+s) & \text{otherwise}.
 \end{cases}\notag
 \end{align}
 \end{lem}
 
 \begin{proof}
Following the proof of \cite[Lemma 4.4]{rcb1}, we find that
 \begin{align*}
\sum_{n\in I} e(\ell f(x)) &- s^{-1} S(s, \ell G) \sum_{n\in I} e(\ell F(n))\\[2mm]
&\ll s^{-1} \sum_{b=1}^{s-1} \left\|\frac bs\right\|^{-1} \Bigg| \sum_{v=1}^s e\left(\frac{\ell G(v) + bv}s\right)\Bigg|\\
\intertext{and}
\sum_{n\in I} e(\ell F(n)) &= \int_I e(\ell F(z)) dz + O(1).
 \end{align*}
Moreover, by a standard estimate \cite{coch},
 \begin{align*}
S(s, \ell G) \ll \langle \ell &, s\rangle^{1/k} s^{1-1/k} N^\eta,\\[2mm]
\sum_{v=1}^s e\left(\frac{\ell G(v) + bv}s\right) &\ll \langle \ell u_1 + b, \ell u_2, \ldots, \ell u_k, s\rangle^{1/k} s^{1-1/k} N^\eta\\[2mm]
&\ll D_\ell s^{1-1/k} N^\eta
 \end{align*}
where $D_\ell = \min(L^2N^\eta, \langle \ell u_1 + b, s\rangle)^{1/k}$ if $J(f) = 1$ and $D_\ell = \langle \ell u_1 + b, s\rangle$ otherwise.

It follows that
 \begin{align*}
\sum_{n\in I} &e(\ell f(x)) - s^{-1} S(s, \ell G)\int_I e(\ell F(z))dz\\[2mm]
&\quad \ll \langle \ell, s\rangle^{1/k} s^{1-1/k} N^\eta + N^\eta s^{-1/k} \sum_{b=1}^{s-1} \left\|\frac bs\right\|^{-1} D_\ell. 
 \end{align*}

We now sum the absolute values of the left-hand side over $\ell \le L$. Since the contribution from $\sum\limits_{\ell = 1}^L \langle \ell, s\rangle^{1/k}$ is $\ll LN^\eta$, a splitting-up argument shows that we need only prove the bound
 \[
\ll \begin{cases}
N^\eta Ls^{1-1/k} & (J(f) = 2^{k-1})\\[2mm]
N^\eta(Ls^{1-1/k}+s) & \text{(otherwise)}
 \end{cases}
 \]
for the quantity
 \begin{equation}\label{eq3.3}
\frac{s^{1-1/k}}B \ \sum_{\frac B2 \le b < 2B} \ \sum_{\substack{\ell = 1\\
d\mid \ell u + b}}^L A_d 
 \end{equation}
where $d\mid s$ and
 \[
A_d = \begin{cases}
\min(L^{2/k} N^\eta, d^{1/k}) & \text{if } J(f) = 1\\[2mm]
d^{1/k} & \text{otherwise}.
 \end{cases}
 \]
Applying \Cref{lem6}, the left-hand side of \eqref{eq3.3} is
 \[
\ll \frac{s^{1-1/k}}B\ \min(L^{2/k} N^\eta, d^{1/k}) \left(\frac{BL}d + B\right) \ll N^\eta s^{1-1/k} L
 \]
if $J(f) = 2^{k-1}$. Otherwise we obtain
 \[
\ll \frac{s^{1-1/k}}B\ d^{1/k} \left(\frac{BL}d + B\right) \ll s^{1-1/k} L + s.
 \]
This proves the lemma.
 \end{proof}

 \begin{lem}\label{lem8}
Let $k \ge 2$ and $Y \ll N^{1-5\rho/2}(k=2)$, $Y \ll N^{1/2+\rho}$ $(k \ge 3)$. Then, with $g$ as defined in Section \ref{sec1}, we have
 \begin{equation}\label{eq3.4}
\sum_{\ell = 1}^L \ \sum_{y\sim Y} \Bigg|\sum_{n\in I(y)} e(\ell g(yn))\Bigg| \ll N^{1-2\eta},
 \end{equation}
where $I(y) = \left(\frac N{2y}, \frac Ny\right]$.
 \end{lem}
 
 \begin{proof}
Let $\mc S$ be the set of $y \sim Y$ with $\langle y, q\rangle \le N^\rho$ and
 \begin{equation}\label{eq3.5}
\sum_{\ell = 1}^L  \Bigg|\sum_{n\in I(y)} e(\ell g(yn))\Bigg| > N^{1-2\eta} Y^{-1}.
 \end{equation}
It suffices to show that
 \[
T:= \sum_{y \in \mc S} \  \sum_{\ell = 1}^L  \Bigg|\sum_{n\in I(y)} e(\ell g(yn))\Bigg| \ll N^{1-2\eta}.
 \]
To see this, the contribution in \eqref{eq3.4} from $y\sim Y$ for which \eqref{eq3.5} fails is $< N^{1-2\eta}$. The contribution from $y \sim Y$ for which $\langle y, q\rangle$ is a fixed divisor $d$ of $q$, $d > N^\rho$, is
 \[
\ll \frac Yd \, L\, \frac NY \ll N^{1-3\eta}
 \]
and our claim follows on summing over $d$.

Given $y \in \mc S$, we apply \eqref{eq3.5} in Lemmas \ref{lem5} and \ref{lem7}. Here $\g_k = ay^k/q$, $\g_j = \a_jy^j$ $(j < k)$. Suppose first that $J(f) = 2^{k-1}$. Take $X = \frac NY$, $M = L$ in \Cref{lem5}. Then
 \[
P = N^{1-2\eta} Y^{-1} \ge L\left(\frac NY\right)^{1-1/J+\eta}
 \]
since $Y \ll N^{1-J\rho}$. The integers $s$, $u_1, \ldots, u_k$ provided by \Cref{lem5} satisfy \eqref{eq3.1}, since $k \le J$ and so
 \[
L^k N^\eta \ll \frac NY\, N^{-\eta}.
 \]
Lemma 7 yields
 \begin{align}
\frac{N^{1-2\eta}}Y &\ll \sum_{\ell = 1}^L \Bigg|\sum_{n\in I(y)} e(\ell g(n))\Bigg|\label{eq3.6}\\[2mm]
&\ll \sum_{\ell=1}^L  \Big|s^{-1} S(s, \ell G) \int_{I(y)} e(\ell F(z))dz\Big| + L^k N^{3\eta k}\notag
 \end{align}
where $I(y) = \left(\frac N{2y}, \frac Ny\right]$. Here we suppress dependence of $F$, $G$ on $y$. The last term is of smaller order than $\frac{N^{1-2\eta}}Y$ in \eqref{eq3.6}, so that
 \begin{equation}\label{eq3.7}
\frac{N^{1-2\eta}}Y \ll \sum_{\ell=1}^L \left|s^{-1} S(s, \ell G) \int_{I(y)} e(\ell F(z))dz\right|.
 \end{equation}

We now show that that \eqref{eq3.7} also holds when $J(f) \ne 2^{k-1}$. Select $m_0 = m_0(y)$ such that
 \[
\Bigg|\sum_{n\in I(y)} e(m_0g(n))\Bigg| \ge P : = \frac{N^{1-2\eta}}{YL}\, . 
 \]
We have $Y \ll N^{1-J\rho}$, as is easily verified. Hence
 \[
P \ge \left(\frac NY\right)^{1-\frac 1J + \eta}.
 \] 
We apply \Cref{lem5} with $f = m_0g$, obtaining integers $s'$, $u_1',\ldots,u_k'$ with $s' \ll L^k N^{3k\eta}$,
 \[
|s' m_0 \g_j - u_j'| \ll L^k N^{3k\eta} \left(\frac NY\right)^{-j} \quad (j=1, \ldots, k).
 \] 

Let $d = \langle s' m_0, u_1', \ldots, u_k'\rangle$ and $s = \frac{s'm_0}d$, $u_j = \frac{u_j'}d$. Then 
 \[
|s\g_j - u_j| \ll L^k N^{3\rho k}\left(\frac NY\right)^{-j} \ll L^{-1} N^{-\eta} \left(\frac NY\right)^{-(j-1)} 
 \]
since $J \ge k + 1$. Thus we can apply \Cref{lem7}. In the analogue of \eqref{eq3.6}, the second term on the right-hand side is now
 \[
\ll (Ls^{1-1/k} + s) N^\eta \ll N^{(k+1)\rho} \ll \frac NY\, N^{-3\eta}, 
 \]
and we again end up with \eqref{eq3.7}.

Factorizing $s$ as $s = s_0s_1$ with $s_0$ square-full, $s_1$ square-free, and $\langle s_0, s_1\rangle = 1$, we have
 \begin{equation}\label{eq3.8}
s^{-1} S(s, \ell G) \ll \left(\frac{s_0}{\langle s_0, \ell\rangle}\right)^{-1/k} \left(\frac{s_1}{\langle s_1, \ell\rangle}\right)^{-1/2}. 
 \end{equation}
See Cochrane \cite{coch} for more general results. The estimate
 \begin{equation}\label{eq3.9}
\int_{I(y)} e(\ell F(z))dz \ll \min \left(\frac NY, \ell^{-1/k}\left|\frac{y^ka}q - \frac{u_1}s\right|^{-1/k}\right) 
 \end{equation}
is a consequence of Vaughan \cite[Theorem 7.3]{vau}. Putting the trivial estimate in \eqref{eq3.9} together with \eqref{eq3.7}, \eqref{eq3.8}, we have
 \[
\frac{N^{1-2\eta}}Y \ll \sum_{\ell=1}^L \langle s, \ell\rangle^{1/2} \cdot \frac 1{s_0^{1/k} s_1^{1/2}} \ \frac NY, 
 \]
whence
 \begin{equation}\label{eq3.10}
s_0^{1/k} s_1^{1/2} \ll LN^{3\eta}. 
 \end{equation}

We now subdivide $\mc S$ into $O((\log N)^3)$ classes according to the values of $s_0 = s_0(y)$, $s_1 = s_1(y)$ and
 \[
\left|\frac{y^ka}q - \frac{u_k}{s_0s_1}\right|.
 \]
In each class $\mc S(Z, S_0, S_1)$, we have $s_0 \sim S_0$, $s_1 \sim S_1$ with
 \begin{equation}\label{eq3.11}
S_0^{1/k} S_1^{1/2} \ll LN^{3\eta}, 
 \end{equation}
and, with $Z_0$ defined below, $Z = 2^{-j}Z_0 \ge 2$, also
  \begin{equation}\label{eq3.12}
\frac 1{2s_0s_1 Z} \le \left|\frac{y^ka}q - \frac{u_k}{s_0s_1}\right| < \frac 1{s_0s_1Z} \ \text{ or } \ (\text{if } Z = Z_0) \left|\frac{y^ka}q - \frac{u_k}{s_0s_1}\right| < \frac 1{s_0s_1Z}.
 \end{equation}
Here
 \[
L^{-1/k}(Z_0 S_0 S_1)^{1/k} = \frac NY\, .
 \]
From \eqref{eq3.7}, \eqref{eq3.8}, \eqref{eq3.9}, \eqref{eq3.11}, \eqref{eq3.12} there is a class $\mc S^* = \mc S(Z, S_0, S_1)$ and an $L_0\in [1,L)$ such that
 \begin{align*}
T &\ll N^\eta \sum_{y \in \mc S^*} \ \sum_{\ell \sim L_0} \langle s_0(y) s_1(y), \ell\rangle^{1/2} L_0^{-1/k} Z^{1/k} S_1^{-\frac 12 + \frac 1k}\\[2mm]
&\ll S_1^{-\frac 12 + \frac 1k} L^{1-1/k} N^{2\eta} Z^{1/k} \# \ \mc S^*.
 \end{align*}

We can estimate $\#\, \mc S^*$ using the results of \Cref{sec2}, since for every $y \in \mc S^*$, there is an $s \in \mc A(S_0, S_1, 1, 1)$ with
 \[
\left\|\frac{say^k}q\right\| < \frac 1Z.
 \]
Thus, in the notation of \Cref{sec2},
 \begin{equation}\label{eq3.13}
T \ll S_1^{-\frac 12 + \frac 1k} L^{1-1/k} N^{2\eta} Z^{1/k} \mc M_k(Y, Z, S_0, S_1). 
 \end{equation}
We now conclude the proof by considering separately the cases $k=2$, $k=3$, and $k\ge 4$. It is easy to verify the condition \eqref{eq2.5} needed for our bounds on $\mc M_k$, since
 \[
S_0 S_1 N^{k\rho} \ll N^{2k\rho} \ll N \ll qN^{-\rho/2}.
 \]

\noindent $\bs{k=2}$. Recalling \eqref{eq2.6}, we deduce from \eqref{eq3.13} that
 \begin{align*}
T &\ll L^{1/2} N^{5\eta} S_0^{1/2} S_1(Y + q^{1/2})\\[2mm] 
&\ll L^{1/2} N^{5\eta} S_0^{1/2} S_1 N^{1-5\rho/2}
 \end{align*}
since $\frac 12 + \frac\rho 4 < 1 - \frac{5\rho}2$. Using \eqref{eq3.11},
 \[
T \ll L^{5/2} N^{1 + 11\eta - 5\rho/2} \ll N^{1-2\eta}.
 \]
\bigskip

\noindent $\bs{k=3}$. Recalling \eqref{eq2.4}, \eqref{eq2.8},
 \begin{align}
T &\ll L^{2/3} N^{4\eta} Z^{1/3} S_0^{1/2} S_1^{5/6} \min\left(\frac qZ, \, Y^{1/2} + \frac{(Y + Y^{1/4} q^{1/4})}{Z^{1/4}}\right)\label{eq3.14}\\[2mm]
&\ll L^{7/3} N^{9\eta} Z^{1/3} \min \left(\frac qZ, \, N^{1/4+\rho/2} + \frac{N^{1/2+\rho}}{Z^{1/4}}\right),\notag
 \end{align}
since $Y^{1/4}q^{1/4} \ll N^{1/8 + \rho/4 + 3/8 + \rho/8}$. Next,
 \begin{align*}
L^{7/3} N^{9\eta} Z^{1/3} &\min\left(\frac qZ, \, N^{1/4+\rho/2}\right)\\[2mm]
&\le L^{7/3} N^{9\eta} q^{1/3} N^{(1/4+\rho/2)2/3}\\[2mm]
&\ll N^{2/3 + \rho(7/3 + 1/6 + 1/3)} \ll N^{1-2\eta}\\
\intertext{and}
L^{7/3} N^{9\eta} Z^{1/3} &\min\left(\frac qZ, \, \frac{N^{1/2+\rho}}{Z^{1/4}}\right)\\[2mm]
&\le L^{7/3} N^{9\eta} q^{1/9} (N^{1/2 + \rho})^{8/9}\\[2mm]
&\ll N^{11/18 + \rho(7/3 + 1/18 + 8/9)} \ll N^{1-2\eta},
 \end{align*}
completing the discussion for $k = 3$.
 \bigskip

\noindent $\bs{k\ge4}$. Using \eqref{eq2.4}, \eqref{eq3.11},
 \begin{align*}
T &\ll S_1^{-1/2 + 1/k} L^{1-1/k} N^{5\eta} Z^{1/k} \min \left(Y, \frac{S_0^{1/2} S_1q}Z\right)\\[2mm]
&\ll L^{1-1/k} N^{5\eta} Y^{1-1/k} S_0^{1/2k} S_1^{-1/2 + 2/k} q^{1/k}\\[2mm]
&\ll N^{\rho(1-1/k)+(1/2+\rho)(1-1/k) + \rho/2 + 1/2 + \rho/2k}\\[2mm]
&\ll N^{1-2\eta} 
 \end{align*}
since (as we easily verify)
 \[
\rho\left(\frac 52 - \frac 3{2k}\right) < \frac 1{2k}.
 \]
This completes the proof of \Cref{lem8}.
 \end{proof}
 \vskip .3in

\section{Type II sums}\label{sec4}

Our object in the present section is to prove

 \begin{lem}\label{lem9}
For $k \ge 3$, let $N^\rho \ll Y \ll N^{1-2J\rho}$,$J = J(f_k)$. Let $|a_x| \le 1$ $(x \le \frac NY)$, $|b_y| \le 1$ $(y \sim Y)$. Then 
 \[
S : =  \sum_{\ell = 1}^L \Bigg|\mathop{\sum_{x \le \frac NY} \, a_x \, \sum_{y\sim Y}}_{\frac N2 < xy \le N} b_y e(\ell g(xy))\Bigg| \ll N^{1-\eta}.
 \]
  \end{lem}
  
We observe that the condition $\frac N2 < xy \le N$ may be removed at the cost of a log factor \cite[Section 3.2]{har4}, and we shall show that $S'$, defined like $S$ without this condition, is $\ll N^{1-2\eta}$.

 \begin{proof}[Proof of \Cref{lem9}]
We write, throughout this section,
 \[
S' = \sum_{\ell = 1}^L c_\ell \sum_{y\sim Y} b_y \sum_{x \le \frac NY} a_x e(\ell g(xy))
 \]
where $|c_\ell| = 1$, so that
 \[
|S'| \le \sum_{x \le \frac NY} \Bigg|\sum_{\ell = 1}^L c_\ell \sum_{y\sim Y} b_y e(\ell g(xy))\Bigg|.
 \]
By the Cauchy-Schwarz inequality,
 \begin{align}
|S'|^2 &\le \frac NY \ \sum_{x\le \frac NY} \Bigg|\sum_{\ell = 1}^L \ \sum_{y\sim Y} \ c_\ell b_y e(\ell g(xy))\Bigg|^2\label{eq4.1}\\[2mm]
&= \frac NY \ \sum_{\ell_1, \ell_2 = 1}^L \ \sum_{y_1, y_2 \sim Y} c_{\ell_1} \bar c_{\ell_2} b_{y_1} \bar b_{y_2} \ \sum_{x\le \frac NY} e(\ell_1g(xy_1) - \ell_2g(xy_2)).\notag
 \end{align}

The contribution from quadruples $(\ell_1, \ell_2, y_1, y_2)$ with $\ell_1 y_1^k = \ell_2 y_2^k$ is
 \[
\ll \left(\frac NY\right)^2 N^\eta LY \ll N^{2-4\eta}
 \]
by a divisor argument, since $Y \ge N^\rho$. Hence it suffices to show that
 \[
\Bigg|\sum_{x \le \frac NY} e(\ell_1 g(xy_1) - \ell_2 g(xy_2))\Bigg| < \frac{N^{1-4\eta}}Y\, L^{-2}
 \]
for a given quadruple with $\ell_1 y_1^k \ne \ell_2 y_2^k$, $\ell_j \le L$, $y_j\sim Y$.

Suppose the contrary. We may apply \Cref{lem5} with $X = \frac NY$, $M=1$, and $P = \frac{N^{1-4\eta}}Y\, L^{-2}$. We have
 \[
P \ge X^{1 - \frac 1J + \eta}
 \]
since
 \[
X^{1-\frac 1J + \eta} P^{-1} \le N^{5\eta}\left(\frac NY\right)^{-\frac 1J} L^2 \le N^{5\eta}(N^{2J\rho})^{-\frac 1J} L^2 \le 1.
 \]
Hence there exists a natural number $s$ and an integer $u$,
 \begin{gather*}
s \ll (N^{2\eta} L^2)^k \ll N^{2k\rho-\eta},\\
\left|s \left(\frac{\ell_1 ay_1^k - \ell_2ay_2^k}q\right) - u\right| < N^{2k\rho} \left(\frac NY\right)^{-k}, 
 \end{gather*}
that is,
 \begin{equation}\label{eq4.2}
\left\|\frac{sa}q\, (\ell_1 y_1^k - \ell_2 y_2^k)\right\| < N^{2k\rho}\left(\frac NY\right)^{-k}. 
 \end{equation}
The right-hand side of \eqref{eq4.2} is less than $1/q$, since $Y \ll N^{1/5}$ and
 \begin{equation}\label{eq4.3}
N^{2k\rho}\left(\frac NY\right)^{-k} q \ll N^{(2k+\frac 12)\rho - 3k/10} \ll N^{-\eta} 
 \end{equation}
(it is easy to verify that $\rho < \frac{3k}{20k+5}$). However, the integer $sa(\ell_1 y_1^k - \ell_2y_2^k)$ is not divisible by $q$, since
 \[
1 \le |s(\ell_1 y_1^k - \ell_2 y_2^k)| \ll N^{(2k+1)\rho} Y^k \ll qN^{-\eta}
 \]
by the same inequality
 \[
\left(2k + \frac 12\right)\rho \le \frac{3k}{10} - \eta
 \]
used in \eqref{eq4.3}. Thus \eqref{eq4.2} cannot hold. This completes the proof of \Cref{lem9}.
 \end{proof}
 
Before we consider a variant of \Cref{lem9} for $k=2$, we note the following way of using the Cauchy-Schwarz inequality:
 \begin{align*}
|S'|^2 &\le \Bigg\{\sum_{\ell = 1}^L \ \sum_{x \le \frac NY} \Bigg|\sum_{y\sim Y} b_y e(\ell g(xy))\Bigg|\Bigg\}^2\\[2mm]
&\le \frac{LN}Y \ \sum_{y_1, y_2 \sim Y} \Bigg|\sum_{\ell=1}^L \ \sum_{x\le \frac NY} e(\ell(g(xy_1) - g(xy_2)))\Bigg|.
 \end{align*}
To prove that $|S'|^2 \ll N^{2-4\eta}$ it suffices to show that
 \begin{equation}\label{eq4.4}
R : = \sum_{y_1, y_2 \sim Y} \ \sum_{\ell=1}^L \Bigg|\sum_{x \le \frac NY} e(\ell(g(xy_1) - g(xy_2)))\Bigg| \ll N^{1-4\eta} YL^{-1}. 
 \end{equation}

We need one more lemma.

 \begin{lem}\label{lem10}
Let $W$, $X$, $Y$ be positive integers greater than 1. Then the inequality
 \[
\min_{s\le W} \left\|\frac{a(y_1^2 - y_2^2)s}q\right\| < \frac 1X
 \]
is satisfied for
 \[
\ll \left(\frac{WY^2}q + 1\right) \left(1 + \frac qX\right) (WY)^\eta
 \]
pairs $y_1$, $y_2$ with $y_1 \ne y_2$, $y_1$, $y_2 \sim Y$.
 \end{lem}
 
 \begin{proof}
See \cite[Lemma 7]{bh}.
 \end{proof}
 
 \begin{lem}\label{lem11}
Let $k=2$. Let $N^{2\rho} \ll Y \ll N^{1-4\rho}$. Then for $|a_x| \le 1$ $\left(x \le \frac NY\right)$, $|b_y| \le 1$ $(y \sim Y)$, we have
 \[
\sum_{\ell =1}^L \Bigg|\sum_{x\le N/Y} \, a_x \sum_{\substack{y\sim Y\\
\frac N2 < xy \le N}} b_y e(\ell g(xy))\Bigg| \ll N^{1-\eta}.
 \]
 \end{lem}
 
 \begin{proof}
As already noted, it suffices to prove \eqref{eq4.4}. Since $Y \gg L^2N^{4\eta}$, we need only consider the contribution to $R$ from pairs $(y_1, y_2)$ with $y_1 \ne y_2$,
 \[
\sum_{\ell=1}^L \Bigg|\sum_{x\le \frac NY} e(\ell(g(xy_1) - g(xy_2)))\Bigg| > N^{1-4\eta} Y^{-1}L^{-1}.
 \]

For such a pair $(y_1, y_2)$, we apply \Cref{lem5}, with $k=2$; $L$, $N/Y$ in place of $M$, $X$; and $P = N^{1-4\eta} Y^{-1}L^{-1}$. We have
 \[
P \ge L\left(\frac NY\right)^{1/2 + \eta}
 \]
since $Y \ll N^{1-4\rho}$. Thus (suppressing dependence on $y_1$, $y_2$) there are natural numbers $s$ and integers $u_1$, $u_2$ with $s \ll N^{4\rho}$, $\langle s, u_2\rangle \le LN^\eta$ and, for $\g_2 = (y_1^2 - y_2^2)\, \frac aq$, $\g_1 = (y_1 - y_2)\a_1$, satisfying
 \begin{equation}\label{eq4.5}
|s\g_j - u_j| \ll L^{-1} \left(\frac NY\right)^{-j} N^{4\rho - 2\eta} \quad (j=1,2). 
 \end{equation}
It is clear that \Cref{lem7} is applicable. Thus
 \begin{align*}
\sum_{\ell = 1}^L \Bigg|\sum_{x\le \frac NY} &e(\ell(\g_2x^2 + \g_1x)) - s^{-1} S(s, \ell G) \int_0^{N/Y}\kern-10pt e(\ell F(z)) dz\Bigg|\ll N^{2\eta} Ls^{1/2}\\[2mm]
&\ll N^{3\rho-5\eta} \ll \frac{N^{1-5\eta}}{YL}.
 \end{align*}
Thus
 \begin{align}
\sum_{\ell = 1}^L \Bigg|\sum_{x\le \frac NY} e(\ell(\g_2x^2 + \g_1x))\Bigg|&\ll \sum_{\ell=1}^L s^{-1/2} \langle s, \ell\rangle^{1/2} \frac NY\label{eq4.6}\\[2mm]
&\ll L s^{-1/2}\, \frac{N^{1+\eta}}Y.\notag
 \end{align}
In particular, for these pairs $(y_1, y_2)$ we have
 \[
s^{1/2} \ll L^2 N^{5\eta}\ ; \ 2s \le N^{4\rho}.
 \]

Let $X = Y^{-2} LN^{2-4\rho}$. Then \eqref{eq4.5} implies
 \begin{equation}\label{eq4.7}
\left\|\frac{sa(y_1^2 - y_2^2)}q\right\| < \frac 1X. 
 \end{equation}
The number of pairs $(y_1, y_2)$ with $s\sim W$ (for $W \le N^{4\rho}$) is
 \[
\ll \left(\frac{WY^2}q + 1\right) \left(1 + \frac qX\right) N^\eta
 \]
by \Cref{lem10}, and these pairs contribute to $R$ an amount
 \begin{align*}
&\ll \frac{LN^{1+2\eta}}{YW^{1/2}} \left(\frac{WY^2}q + 1\right)\left(1 + \frac qX\right)\\
\intertext{(by \eqref{eq4.6})}
&\ll \frac{W^{1/2}LN^{1+2\eta}Y}q + \frac{W^{1/2}L N^{1+2\eta}Y}X + \frac{LN^{1+2\eta}}Y + \frac{LN^{1+2\eta}q}{XY}.
 \end{align*}
Each of these four terms is $\ll N^{1-5\eta}YL^{-1}$:
 \begin{align*}
\frac{W^{1/2}LN^{1+2\eta}Y}q &\ll N^{5\rho/2} Y \ll N^{1-5\eta} YL^{-1} \quad (\text{since } \rho < 2/7);\\[2mm]
\frac{W^{1/2}LN^{1+2\eta}Y}X &=  W^{1/2} N^{-1+ 4\rho+2\eta} Y^3 \ll N^{-1+6\rho+2\eta} Y^3 \ll \frac{N^{1-5\eta}Y}L\\
\intertext{(since $Y \ll N^{1-4\rho}$)}
\frac{LN^{1+2\eta}}Y &\ll \frac{N^{1-5\eta}Y}L \quad (\text{since } Y \gg N^{2\rho});\\
\intertext{and}
\frac{LN^{1+2\eta}q}{XY} &= qN^{-1+4\rho + 2\eta} Y \ll N^{\frac{9\rho}2 +2\eta} Y \ll \frac{N^{1-5\eta}Y}L
 \end{align*}
(since $\rho < \frac 2{11}$).
  \bigskip

We now sum over $O(\log N)$ values of $W = N^{4\rho}2^{-j}$ and obtain the desired bound \eqref{eq4.4}. This completes the proof of \Cref{lem11}.
 \end{proof}
 \vskip .3in

\section{Application of the Harman sieve.}\label{sec5}

We use the standard notations
 \begin{gather*}
P(z) = \prod_{p < z} p ,\\[2mm]
E_d = \{n : dn \in E\} \quad \text{for a finite subset $E$ of $\mb N$, while}\\[2mm]
S(E, z) = \sum_{\substack{n\in E\\[1mm] \langle n, P(z)\rangle = 1}} 1, \quad \chi_E = \text{indicator function of $E$}.
 \end{gather*}
We `compare' the set $A$ introduced in \Cref{sec1}, $(A = A(f_k), J = J(f_k))$ with the set
 \[
B = \left\{n \in \mb N : \frac N2 < n \le N\right\}.
 \]
We write $[\a, \a + \b] = [2\rho, 1 - 5\rho/2]$ $(k = 2)$; $[\a, \a+\b] = [\rho, 1 - 2J\rho]$ $(k > 2)$.

 \begin{lem}\label{lem12}
Let $u_h$ $(h \le H)$ be real numbers with $|u_h| \le 1$, $u_h = 0$ for $(h, P(N^\eta)) > 1$. Suppose that
 \begin{equation}\label{eq5.1}
H < MN^{-\a} 
 \end{equation}
where $M \ll N^{1-5\rho/2}$ for $k=2$ and $M \ll N^{1/2+\rho}$ for $k \ge 3$. Then, writing $\d$ for $L_1^{-1}$,
 \[
\sum_{h \le H} u_h S(A_h, N^\b) - 2\d \sum_{h \le H} u_h S(B_h, N^\b) \ll \d N^{1-\eta/2}.
 \]
 \end{lem}
 
 \begin{proof}
We first apply \cite{bw}, Lemma 14, which is a variant of \cite[Theorem 3.1]{har4} convenient for our purposes. By choosing the weight function of the lemma to be
 \[
w(n) = \chi_A(n) - 2\d \chi_B(n),
 \]
we see that we need only show that
 \begin{align*}
&\sum_{\substack{\frac N2 < mn \le N\\
m \le M}} a_m (\chi_A(mn) - 2\d \chi_B(mn)) \ll \d N^{1-2\eta/3},\\[2mm]
&\sum_{N^\a \le m \le N^{\a+\b}} a_m \sum_n b_n(\chi_A(mn) - 2\d \chi_B(mn)) \ll \d N^{1-2\eta/3}
 \end{align*}
where the $a_m$, $b_n$ are arbitrary with $|a_m| \le 1$, $|b_n| \le 1$.  Now a standard use of upper and lower bounds for the indicator function of $\chi_A$ reduces our task to proving bounds for exponential sums that have already been given in \Cref{lem8} and \Cref{lem9}; compare, for example, the arguments in \cite[Section 3.4]{har4}.
 \end{proof}
 
Another `comparison' that will be needed below is a bound $O(\d N^{1-2\eta/3})$ for a sum
 \[
\sum_{N^\a \le p \le N^{\a+\b}} (S(A_p, p) - 2\d S(B_p, p))
 \]
This reduces to a sum of $O(1)$ expressions, in which $r = O(1)$, of the form
 \[
\sum_{\substack{N^\a \le p \le N^{\a+\b}, \, \frac N2 < pp_1\ldots p_r \le N\\[1mm]
p\le p_1 \le p_2 \le \cdots \le p_r}} \{\chi_A(pp_1 \ldots p_r) - 2\d \chi_B(pp_1\ldots p_r)\}.
 \]
\Cref{lem9} provides a satisfactory estimate, using \cite[Section 3.2]{har4} to remove the condition $p \le p_1$ that will occur in the Type II exponential sums that arise.

 \begin{proof}[Proof of \Cref{thm1} for $k \ge 3$].
We observe that
 \begin{align}
\#\{p:p \in A\} &= S(A, (2N)^{1/2})\label{eq5.2}\\[1mm]
&= S(A, N^\a) - \sum_{N^\a \le p \le (2N)^{1/2}} S(A_p, p)\notag
 \end{align}
by an application of Buchstab's identity. Iterating the procedure,
 \[
\#\{p:p \in A\} = S_1 - S_2 - S_3 + S_4
 \]
where
 \begin{align*}
S_1 &= S(A, N^\a), \quad S_2 = \sum_{N^\a \le p \le N^{\a+\b}} S(A_p, p),\\[2mm]
S_3 &= \sum_{N^{\a+\b} < p \le (2N)^{1/2}} \Bigg(S(A_p, N^\a) - \sum_{N^\a \le q < N^{\a+\b}} S(A_{pq},q)\Bigg)\\[2mm]
S_4 &= \sum_{\substack{N^{\a+\b} < p \le (2N)^{1/2}\\
N^{\a+\b} \le q <p}} S(A_{pq}, q).
 \end{align*}

Define $S_1'$, $S_2'$, $S_3'$, and $S_4'$ in the same way as $S_1, \ldots, S_4$, with $A$ replaced by $B$. We observe that
 \begin{align*}
\#\{p : p \in A\} &\ge S_1 - S_2 - S_3\\
&= 2\d(S_1' - S_2' - S_3') + O(\d N^{1-\eta/2})
 \end{align*}
by applying \Cref{lem12}, and the remarks following the proof of \Cref{lem12}, to $S_1$, $S_2$, and $S_3$. (The condition \eqref{eq5.1} will be trivial for $S_1$, and will amount to
 \[
(2N)^{1/2} \le MN^{-\rho}
 \]
for $S_3$, where $M= 2^{1/2}N^{1/2+\rho}$.) We now follow arguments familiar from \cite{har4}. To show that $S_1' - S_2' - S_3' > b(N\d/\log N)$ for a small positive $b$, it suffices to show that, $\omega$ denoting the Buchstab function,
 \begin{equation}\label{eq5.3}
\iint_{0.1842 <y < x < 1/2, \, x + 2y < 1} \omega \left(\frac{1-x-y}y\right) \frac{dx}x\ \frac{dy}{y^2} < 1.
 \end{equation}
(Here we use $\a + \b = 2/10$ for $k=3$ and $\a + \b = 0.1842$ for $k\ge 4$.)  The constant 0.1842 is close to sharp (within $10^{-4}$). The validity of \eqref{eq5.3} was kindly checked by Andreas Weingartner.
 \end{proof}
 
 \begin{proof}[Proof of \Cref{thm1} for $k=2$.]
We now have
 \[
[\a, \a + \b] = \left[\frac 4{13}, \frac 5{13}\right]
 \]
and the condition \eqref{eq5.1} reduces to
 \[
H \ll N^{8/13}.
 \]
We can weaken this restriction on $H$ to
 \[
H \ll N^{9/13}
 \]
by treating the Type I exponential sums in question as Type II exponential sums with one variable between $N^{1-5/13}$ and $O(N^{1-4/13})$. Now we have stronger `arithmetic information' than is used by Harman \cite[Section 5.3]{har4} in the discussion of the Diophantine inequality.
 \[
\|p\a + \b\| < p^{-0.3182} \quad \left(\text{since } \frac 2{13} < \frac{0.3182}2\right).
 \]
We can follow the proof there verbatim to obtain \Cref{thm1}.
 \end{proof}

\hskip .5in

 \end{document}